\documentclass[11pt,reqno]{amsart}

\usepackage[margin=1.38in]{geometry}
\usepackage{amsmath,graphicx,amssymb,amsthm,fixltx2e,enumerate,marvosym,amsfonts,graphics,cite,array,bm}

\newtheorem{theorem}{Theorem}[section]
\newtheorem{problem}[theorem]{Problem}
\newtheorem{lemma}[theorem]{Lemma} 
 
\newtheorem{corollary}[theorem]{Corollary}

\newtheoremstyle{backref}% We need this to give our Section 7 lemma the right numbers
  {3pt}{3pt}		% Spacing above and below; I have no idea how close this is to the standard
  {\itshape}		% body font 
  {0pt}{\bfseries}	% the first arg on this line is the indent amount, second is heading font
  {.}				% Punctuation after the theorem name
  { }				% Amount of space after the heading
  {\thmname{#1} \thmnote{#3}}  % Typesets only the name and note, so we can ignore the new numbers

\theoremstyle{backref}

\newcommand{\abs}[1]{\left|{#1}\right|}

\newcommand{\pr}[1]{\left({#1}\right)} 
\newcommand{\piece}[4]{\left\{\begin{array}{cc} {#1} & \textrm{if }{#2} \\ {#3} & \textrm{if }{#4} \end{array}\right.}

\renewcommand{\b}[1]{\mathbf{{#1}}}

\newcommand{\m}[1]{\mathcal{{#1}}}
\newcommand{\cl}[1]{\overline{{#1}}}

\renewcommand{\t}[1]{\textnormal{{#1}}}

\newcommand{\sub}{\subseteq}

\newcommand{\N}{\mathbb{N}}

\renewcommand{\d}{\delta}

\renewcommand{\o}{\omega}

\newcommand{\np}{\newpage}

\renewcommand{\i}{^{-1}}

\newcommand{\ld}{\ldots}

\renewcommand{\l}{\lambda}

\newcommand{\iso}{\cong}

\newcommand{\D}{\Delta}

\renewcommand{\N}{\mathbb{N}}
\newcommand{\sq}{\square_q}

\newcommand{\sm}{\setminus}

\newcommand{\wt}[1]{\widetilde{{#1}}}

\title{Designs from Paley graphs and Peisert graphs}
\author[James Alexander]{James Alexander}
\address[J.~Alexander]{University of Delaware, Newark DE 19716}
\email[J.~Alexander]{jamesja@udel.edu}
\begin{document}
\maketitle

\begin{abstract}
Fix positive integers $p,q,$ and $r$ so that $p$ is prime, $q=p^r$, and $q\equiv 1$ (mod $4$). Fix a graph $G$ as follows: If $r$ is odd or $p\not\equiv 3$ (mod $4$), let $G$ be the $q$-vertex Paley graph; if $r$ is even and $p\equiv 3$ (mod $4$), let $G$ be either the $q$-vertex Paley graph or the $q$-vertex Peisert graph. We use the subgraph structure of $G$ to construct four sequences of $2$-designs, and we compute their parameters. Letting $k_4$ denote the number of $4$-vertex cliques in $G$, we create $62$ additional sequences of $2$-designs from $G$, and show how to express their parameters in terms of only $q$ and $k_4$. We find estimates and precise asymptotics for $k_4$ in the case that $G$ is a Paley graph. We also explain how the presented techniques can be used to find many additional $2$-designs in $G$. All constructed designs contain no repeated blocks. 
\end{abstract}

\section{Introduction and main results}
\label{sec:introduction}

Let $\N$ be the set of positive integers. For the duration of this paper, assume that $p,q,r\in\N$ are fixed so that $p$ is prime, $q=p^r$, and $q\equiv 1$ (mod $4$). Fix a graph $G$ as follows: If $r$ is odd or $p\not\equiv 3$ (mod $4$), let $G$ be the $q$-vertex Paley graph, i.e., the graph with vertex set $\t{GF}(q)$ and edge set consisting of all $\{x,y\}\sub \t{GF}(q)$ so that $x\neq y$ and $x-y$ is a quadratic residue (see \cite{PALEYPATCHONE,PALEYPATCHTWO}); if $r\t{ is even}$ and $p\equiv 3$ (mod $4$), let $G$ be either the $q$-vertex Paley graph or the $q$-vertex Peisert graph, i.e., the graph with vertex set $\t{GF}(q)$ and edge set consisting of, for some fixed primitive root $\o$ of $\t{GF}(q)$,  all $\{x,y\}\sub \t{GF}(q)$ which satisfy that $x-y=\o^j$ for $j\equiv 0$ (mod $4$) or $j\equiv 1$ (mod $4$) (see Section $1$ of \cite{PEISERT}). It is straightforward to check that  this Peisert graph construction does not depend on the choice of  primitive root $\o$.  Also, for the duration of this paper, let $V$ denote the set of vertices of $G$ and $E$ denote the set of edges of $G$. We note that if $p\equiv 3$ (mod $4$) and $q\neq 9$, then the $q$-vertex Paley graph and $q$-vertex Peisert graph will be non-isomorphic by Lemma $6.2$ of \cite{PEISERT}; however, what follows will hold for either graph.

For design theory terminology and notation, we follow Chapter $19$ of \cite{VANWIL}. We show that the subgraph structure of $G$ provides blueprints for constructing many different $2$-designs, none of which contain repeated blocks. Our results will all be based on the fact that $G$ is a very symmetric structure, as explained in Section~\ref{PROPS}. Using graphs with very symmetric properties to construct designs and vice versa is not a new idea (see \cite{GRAPHDESA,GRAPHDESB,GRAPHDESC,GRAPHDESD,GRAPHDESE,GRAPHDESF,GRAPHDESG,GRAPHDESH,GRAPHDESI,GRAPHDESJ,GRAPHDESK,GRAPHDESL,GRAPHDESM} for examples). The particular techniques presented here, however, differ from any of which I am aware. In order to present them, we will need to develop some additional notation.

For basic graph theory terminology, we follow \cite{WEST}. For any $x\in V$, let $N_G(x)$ denote the set of neighbors of $x$ in $G$, i.e., the set of all $y\in V$ so that $\{x,y\}\in E$. For any $X\sub V$, let $G[X]$ denote the subgraph of $G$ induced by $X$, i.e., the graph with vertex set $X$ and edge set consisting of all $\{x,y\}\in E$ satisfying $x,y\in X$. For any graph $H$, let $H(G)$ be defined as the set of all $X\sub V$ so that $G[X]\iso H$. For any $t\in\N$ and $x_1,\ld,x_t\in V$, let $H(G,\{x_1,\ld,x_t\})$ denote the set of $B\in H(G)$ satisfying $x_1,\ld,x_t\in B$. Let $\wt{H}(G):=H(G)\cup\cl{H}(G)$, where $\cl{H}$ denotes the complement of graph $H$ (i.e., the graph with the same vertex set as $H$ and edge set consisting of all $2$-sets of vertices which are not edges in $H$), and let $\wt{H}(G,\{x_1,\ld,x_t\}):=H(G,\{x_1,\ld,x_t\})\cup\cl{H}(G,\{x_1,\ld,x_t\})$. Our first main result is the following, which outlines a general technique for obtaining designs in $G$.

\begin{theorem}\label{thm:backbone}
 For any $k,m\in\N$ and distinct $k$-vertex graphs $H_1,\ld,H_m$ with $H_i\neq\cl{H_j}$ for all distinct $i,j\in[m]$,  if $\m{B}:=\cup_{i=1}^m\wt{H_i}(G)$ then $(V,\m{B})$ forms a $2$-$(q,k,\l)$-design, where 
\begin{equation}
\l=\binom{k}{2}\binom{q}{2}\i\sum_{i=1}^m|\wt{H_i}(G)|. \label{eq:lambda}
\end{equation}
\end{theorem}•

In order to discuss some specific designs given by Theorem~\ref{thm:backbone}, for any $s,t\in\N$, let $K_t$ denote the complete $t$-vertex graph, let $P_t$ denote the $t$-vertex path graph, let $K_{s,t}$ denote the complete bipartite graph with partite sets of sizes $s$ and $t$, let $C_t$ denote the $t$-vertex cycle graph, let $D$ denote the graph which is obtained by removing an edge from the complete $4$-vertex graph (sometimes called the diamond graph), and let $R$ denote the graph which is not a $4$-vertex cycle, but which is obtained by removing two edges from the complete $4$-vertex graph (sometimes called the paw graph). An application of Theorem~\ref{thm:backbone} to some subgraphs of small sizes, together with some counting arguments, will provide the following result:

\begin{corollary}\label{thm:littleguy}
Let $\m{B}_1:=\wt{K_3}(G) $, $\m{B}_2:=\wt{P_3}(G) $, $\m{B}_3:=\wt{P_4}(G)\cup \wt{R}(G)\cup \wt{D}(G)$, and $\m{B}_4:=\wt{K_4}(G)\cup\wt{K_{1,3}}(G)\cup \wt{C_4}(G)$. Then,
\begin{enumerate}[(i)]
\item $(V,\m{B}_1)$ forms a $2$-$(q,3,\l_1)$-design, where $\l_1=(1/4) (q-5)$; \label{fone}
\item $(V,\m{B}_2)$ forms a $2$-$(q,3,\l_2)$-design, where $\l_2=(3/4) (q-1)$;\label{ftwo}
\item  $(V,\m{B}_3)$ forms a $2$-$(q,4,\l_3)$-design, where $\l_3= (3/8) (q-1) (q-3)$;\label{fthree}
\item  $(V,\m{B}_4)$ forms a $2$-$(q,4,\l_4)$-design, where $\l_4= (1/8) (q-3) (q-5)$.\label{ffour}
\end{enumerate}
\end{corollary}•

It is not surprising that there exist designs with these parameters; the existence of designs with parameters that match those given in (\ref{fone})-(\ref{ffour}) are, for example,  guaranteed by the main result of \cite{INJOURN}. To the best of my knowledge, however, there do not exist any explicit constructions of designs whose parameters match the parameters of these designs presented in Corollary~\ref{thm:littleguy}. 

Certainly Theorem~\ref{thm:backbone} provides a blueprint for constructing many more designs than those which appear in Corollary~\ref{thm:littleguy}. Even by choosing different combinations of $\wt{K_4}(G)$, $\wt{D}(G)$, $\wt{R}(G)$, $\wt{C_4}(G)$, $\wt{K_{1,3}}(G)$, and $\wt{P_4}(G)$, we have clear directions for how to construct $64$ different $2$-designs (two of which are trivial, and two of which are Corollary~\ref{thm:littleguy}(\ref{fthree}) and Corollary~\ref{thm:littleguy}(\ref{ffour})). The difficulty in writing down the parameters of these $64$ designs, however, comes from trying to count $|\wt{H}(G)|$ when $H$ is one of these $4$-vertex graphs. The following theorem shows that if we can determine $|K_4(G)|$, then we can immediately compute all of these graph parameters, and thus the parameters of all $64$ designs. Let $k_4:=|K_4(G)|$. 

\np

\begin{theorem}\label{almostfour}
We have
\begin{enumerate}[(a)]
\item $|\wt{K_4}(G)|=2|\cl{K_4}(G)|=2k_4$;
\item $|\wt{D}(G)|=2|D(G)|=2|\cl{D}(G)|= (1/64) q (q-1) (q-5) (q-9) - 12k_4$;
\item $|\wt{R}(G)|=2|R(G)|=2|\cl{R}(G)|= (1/4) q (q-1) (q-5) +24k_4 $;
\item $|\wt{C_4}(G)|=2|C_4(G)|=2|\cl{C_4}(G)|= (1/16) q (q-1) (q-5) + 6k_4$;
\item $|\wt{K_{1,3}}(G)|=2|K_{1,3}(G)|=2|\cl{K_{1,3}}(G)|= (1/96) q (q-1) (q-5) (q-9)- 8k_4$;
\item $|\wt{P_4}(G)|=|P_4(G)|= (1/64) q (q-1) (q^2-10 q+41) - 12k_4$.
\end{enumerate}
\end{theorem}•

We note that $|\wt{P_4}(G)|=|P_4(G)|$ because $P_4\iso\cl{P_4}$, while no other graph mentioned in Theorem~\ref{almostfour} satisfies this property. In addition to giving us design parameters in terms of $k_4$,  Theorem~\ref{almostfour} gives us a way to express $|H(G)|$ in terms of $k_4$ and $q$ for any $4$-vertex graph $H$, since the exactly $11$ non-isomorphic $4$-vertex graphs all appear in this theorem. We note, as discussed further in Section~\ref{sec:kfour}, that Theorem~\ref{almostfour} could be presented in terms of $|H(G)|$ for any $4$-vertex graph $H\neq K_4$ as well, but that it is desirable to write the result as we do because studying $|K_t(G)|=|\cl{K_t}(G)|$ for different values of $t$ in the case when $p=q$ has been a topic of much interest for some time (see\cite{BOLRG,CLIQUEA,CLIQUEB,CLIQUEC,CLIQUED,CLIQUEE,CLIQUEF,CLIQUEG,CLIQUEH,CLIQUEI} for examples). This result adds the extra incentive of $60$ new nontrivial design constructions to that body of work. 

In Section~\ref{sec:kfour}, we discuss the problem of determining the value of $k_4$, and explain that such problems can be very difficult in general. We also tabulate the values of $k_4$ for some small values of $q$ in Section~\ref{sec:kfour}, and we explain that even though the value of $|K_3(G)|$ is independent of whether $G$ was chosen to be the Paley graph or Peisert graph, the value of $k_4$ is not. In Section~\ref{sec:appendix}, we tabulate the parameters of all $64$ designs that can be obtained once $k_4$ is known (using Theorem~\ref{almostfour}) if we assume that $q=29$. In Section~\ref{sec:kfp} we obtain estimates for the value of $k_4$ if $G$ is assumed to be the Paley graph using some known results about the distribution of edges within the Paley graph, and we use these estimates to show that $k_4\sim q^4/1536$ as $q\to\infty$. We also provide an exact expression for $k_4$ in terms of sums of quadratic residue characters over $\t{GF}(q)$, but provide little simplification for this sum.

The proofs of Theorem~\ref{thm:backbone}, Corollary~\ref{thm:littleguy}, and Theorem~\ref{almostfour}, along with some secondary results, will be presented in Section~\ref{PROOFS} once some necessary results concerning the symmetric properties of G are established in Section~\ref{PROPS}. Once $k_4$ has been thoroughly studied in Section~\ref{sec:kfour} and Section~\ref{sec:kfp} as previously discussed, in Section~\ref{sec:bigk}, we address the obvious fact that most of the designs guaranteed by Theorem~\ref{thm:backbone} are concerned with subgraphs on many more than $4$ vertices, and we discuss how many designs we can create from $G$. Finally, in Section~\ref{sec:closing}, we offer some concluding remarks. 

%----  -------  ------  ------------- ---------------  --------------

\section{Properties of $G$}
\label{PROPS}

Before we can prove the results presented in Section~\ref{sec:introduction}, we must better understand the structure of $G$. It turns out that $G$ is a very symmetric structure, as captured in the following theorem, which is part of the main result of \cite{PEISERT}.

\begin{theorem}[Peisert, \cite{PEISERT}]\label{thm:selfcomp}
$G$ is self-complementary, vertex-transitive, and edge-transitive.
\end{theorem}

Before moving on, let us make some notes concerning Theorem~\ref{thm:selfcomp}. First, let us recognize that even though it was in \cite{PEISERT} that this result was first established for all relevant $G$, this result was established in the case that $G$ is the Paley graph previously, with the self-complementary property proven in \cite{PALEYPATCHONE} and the transitive properties proven in \cite{PALEYPATCHTWO}. Let us also note that not only does G satisfy these properties, but in almost all cases, it is the only graph that does. More specifically, for any $n\in\N$, it is shown in \cite{PEISERT} that if there exists an $n$-vertex graph $H$ which is self-complementary, vertex-transitive, and edge-transitive, then $n=p_0^{r_0}$ for some prime $p_0$, and either $r_0$ is odd or $p_0\equiv 1$ (mod $4$) and $H$ is the $q$-vertex Paley graph, or $r_0$ is even, $p_0\equiv 3$ (mod $4$), and $H$ is either the $q$-vertex Peisert graph or $q$-vertex Paley graph, as long as $(p_0,r_0)\neq(23,2)$. In the case that $(p_0,r_0)=(23,2)$, $H$ may be either the Peisert graph, or another graph which we do not discuss here. This converse too was known in the case when $p=q$ previously; this was proven in \cite{PEISTWO} using the main results of \cite{SIXER}. 

It will be convenient to state and prove a lemma in the language of strongly regular graphs (see Section 8.6 of \cite{WEST}). As is standard, for $v,k,\l,\mu\in\N\cup\{0\}$, we call $G$ a $(v,k,\l,\mu)$ strongly regular graph if, for distinct $x,y\in V$, we have that $|N_G(x)|=k$, that $\{x,y\}\in E$ implies $|N_G(x)\cap N_G(y)|=\l$, and that $\{x,y\}\notin E$ implies $|N_G(x)\cap N_G(y)|=\mu$. As discussed in \cite{PEISERT}, it is not difficult to see that since $G$ is self-complementary, vertex-transitive, and edge-transitive, $G$ must be a strongly regular graph. In the case that $q=p$, it is well-known that $G$ is a $(p,(1/2) (p-1),(1/4) (p-5),(1/4) (p-1))$ strongly regular graph (see, for example, \cite{BOLRG}). With a similar proof to the one provided for the $r=1$ case in \cite{BOLRG}, we extend the result as follows: 

\begin{lemma}\label{thm:srg}
$G$ is a $(q,(1/2) (q-1),(1/4) (q-5),(1/4) (q-1))$ strongly regular graph. 
\end{lemma}•

\begin{proof}
As mentioned in the paragraph immediately preceding the statement of this lemma, there are $k,\l,\mu\in\N\cup\{0\}$ so that $G$ is a $(q,k,\l,\mu)$ strongly regular graph. The self-complementary property of $G$ guaranteed by Theorem~\ref{thm:selfcomp} implies that $k=(1/2) (q-1)$. Consider $x\in V$, let $N:=N_G(x)$, and let $M:=V\sm(N_G(x)\cup\{x\})$. By definition, for $y\in N$ and $z\in M$, we have $|N_G(y)\cap N|=\l$ and thus $|N_G(y)\cap M|=k-\l-1$, and $|N_G(z)\cap N|=\mu$. It follows that $|M|m=|N|(k-\l-1)$, and thus 
\begin{equation}
m=k-\l-1=(1/2)(q-3)-\l. \label{into}
\end{equation}
Moreover, since $|V|=q$ and 
\begin{align}
|V\sm\{x,y\}|&=|(N\sm N_G(y))|+|(N_G(y)\sm N)|+\l + \mu\\
	&=2(k-\l-1)+\l+\mu, 
\end{align}
we have 
\begin{equation}
\mu=\l+1. \label{outof}
\end{equation}
A combination of (\ref{into}) and (\ref{outof}) gives the claimed values of $\l$ and $\mu$. 
\end{proof}•

Before ending this section, we will derive from Lemma~\ref{thm:srg} one more property of these graphs which will not be used toward the main objectives of this work, but which may be of particular interest to some, as it contributes to an area of graph theory which has been receiving a lot of recent attention. For this, we need the following definition: Two graphs are called cospectral (or isospectral) if they share the same graph spectrum (i.e., if their adjacency matrices have the same eigenvalues). For well-motivated reasons which we will not address here, finding cospectral graphs has been a goal of many graph theorists lately (see, for examples, \cite{COSPA,COSPB,COSPC,COSPD,COSPE,COSPF,COSPG,COSPH,COSPI,COSPJ}). Now, it is well-known that the spectrum of a strongly regular graph is completely determined by its parameters (see, for example, Section~8.6 of \cite{WEST}); so, since Paley graphs and Peisert graphs are strongly regular graphs which have the same parameters (when considered on the same number of vertices)  by Lemma~\ref{thm:srg}, it follows that Paley graphs and Peisert graphs have the same spectrum. Since Paley graphs and Peisert graphs are always non-isomorphic when considered on more than $9$ vertices by Lemma $6.2$ of \cite{PEISERT}, we have the following:

\begin{corollary}\label{spectralequality}
Paley graphs and Peisert graphs form an infinite sequence of non-isomorphic co-spectral graphs.
\end{corollary}•

Other than to construct infinitely many non-isomorphic cospectral graphs, one could use Corollary~\ref{spectralequality} to show that Peisert graphs enjoy some of the same desirable properties that have made the Paley graph so ubiquitous. For example, some of the pseuorandom properties (see \cite{PSEUDOKB} for definitions) of Paley graphs are traditionally shown using the spectrum of the Paley graph (see, for example, \cite{PSEUDOKB}). Thus, one can establish similar psudorandom properties for the Peisert graph using Corollary~\ref{spectralequality}. We note that some pseudorandom properties of the Peisert graph were proven in \cite{PSEUDOPEIS}, but spectral methods were not used. 
%----  -------  ------  ------------- ---------------  --------------

\section{Proof of the results presented in Section~\ref{sec:introduction}}
\label{PROOFS}

For the duration of this section, we extend our definitions for $V$, $E$, and $N_G(x)$ to any graph $H$ by letting $V(H)$ denote the set of vertices of $H$, $E(H)$ denote the set of edges of $H$, and $N_H(x)$ denote the set of vertices to which $x$ is adjacent in $H$, for any $x\in V(H)$. Also, as is quite standard, let $[n]:=\{1,\ld,n\}$ for any $n\in\N$. All of our proofs rely heavily on the results presented in the previous section, and all involve directly counting subgraphs. Before moving to our proofs, it will be useful to make three simple observations which we will use multiple times, and then to define some notation which will be helpful throughout this section. First, let us notice that it follows immediately from Lemma~\ref{thm:srg} that 
\begin{equation}
|E| = (1/4) q (q-1). \label{edges}
\end{equation}
Next, let us note that the self-complementary property of $G$ guaranteed by Theorem~\ref{thm:selfcomp} gives us that $|H(G)|=|\cl{H}(G)|$ for any graph $H$, and therefore that 
\begin{equation}
|\wt{H}(G)|=
\piece{|H(G)|}{H=\cl{H}}{2|H(G)|}{H\neq\cl{H}}.\label{eq:tildesize}
\end{equation}
Finally, let us note that the edge-transitivity of $G$ guaranteed by Theorem~\ref{thm:selfcomp} implies that
\begin{equation}
|E(H)||H(G)|=|E||H(G,\{x,y\})|.\label{eq:transmain}
\end{equation}

We now move to the main objectives of this section, beginning with the proof of Theorem~\ref{thm:backbone}.

%---------

\begin{proof}[Proof of Theorem~\ref{thm:backbone}]
Let $H_1,\ld,H_m$, $\m{B}$, and $\l$ be defined as in the statement of Theorem~\ref{thm:backbone}, and consider distinct $x,y\in V$. To prove what is desired, it is sufficient to 
show that for any $i\in[m]$, 
\begin{equation}
|\wt{H_i}(G,\{x,y\})|=\binom{k}{2}\binom{q}2\i|\wt{H_i}(G)|\label{eq:goalsofbackbones}
\end{equation}
(since this certainly implies that $x$ and $y$ are contained in exactly $\l$ elements of $\m{B}$ by our definition of $\l$). To that end, first assume that $\{x,y\}\in E$ and that $H_i$ is non-self-complementary.  
It follows from this and the self-complementary property of $G$ guaranteed by Theorem~\ref{thm:selfcomp} that 
\begin{align}
				|\wt{H_i}(G,\{x,y\})|&=|H_i(G,\{x,y\})|+|\cl{H_i}(G,\{x,y\})|\label{startthatthing}\\
				&=\frac{|E(H_i)||H_i(G)|}{|E(G)|}+\frac{|E(\cl{H_i})||\cl{H_i}(G)|}{|E(G)|}\\
				&=\binom{k}{2}\frac{1}{|E|}|H_i(G)|.\label{eq:endthestring}
\end{align}
It follows from (\ref{eq:tildesize}) that
\begin{equation}
\frac{1}{|E|}|H_i(G)|=\binom{q}{2}\i|\wt{H_i}(G)|, \label{eq:sidearm}
\end{equation}
and combination of this with (\ref{startthatthing})-(\ref{eq:endthestring}) gives (\ref{eq:goalsofbackbones}) in this case. 

If $H_i$ is assumed self-complementary, then a similar argument establishes that 
\begin{equation}
|H_i(G,\{x,y\})|=\binom{k}{2}\frac{1}{2|E|}|H_i(G)|,
\end{equation}
and the desired result follows from (\ref{eq:tildesize}). Finally, if $\{x,y\}\notin E$, then $\{x,y\}\in E(\cl{G})$, and the previous arguments apply in every case since $\cl{G}$ is isomorphic to $G$. 
\end{proof}•

%---------

We now move to proving Corollary~\ref{thm:littleguy}. For organizational purposes, we separate the proof into pieces over the four assertions, (\ref{fone})-(\ref{ffour}). Before beginning these proofs, we note that for each $i\in[4]$, it follows from Theorem~\ref{thm:backbone} that $(V,\m{B}_i)$ is a $2$-$(q,k_i,\l_i)$-design where $k_i=3$ when $i\in[2]$ and $k_i=4$ when $i\in\{3,4\}$, so we need only show that $\l_i$ has the claimed value in each case.

%---------

\begin{proof}[Proof of Corollary~\ref{thm:littleguy}(\ref{fone})]
Fix $\{x,y\}\in E$, and let $N_x:=N_G(x)$ and $N_y:=N_G(y)$. It follows from (\ref{eq:transmain}) that
\begin{equation}
3|K_3(G)|=|E||K_3(G,\{x,y\})|\label{triang}.
\end{equation}
By  Lemma~\ref{thm:srg}, we have that 
\begin{equation}
|K_3(G,\{x,y\})|=|N_x\cap N_y|=(1/4) (q-5),\label{mediumsmediums}
\end{equation}
and a combination of this, (\ref{edges}),  and (\ref{triang}) gives
\begin{equation}
|K_3(G)|= (1/48) q (q-1) (q-5)\label{triangtwo}
\end{equation}
and thus that $\l_1$ has the claimed value by (\ref{eq:lambda}) and (\ref{eq:tildesize}). 
\end{proof}•

%---------

While proving Corollary~\ref{thm:littleguy}(\ref{fone}), we have shown that the number of triangles in $G$ is $(1/48) q (q-1) (q-5)$. Specializing this to the case where $q$ is prime, we obtain a short proof that the $p$-vertex Paley graph contains $(1/48) p (p-1) (p-5)$ triangles, thus duplicating one of the main results of \cite{CLIQUEI}, and providing a shorter proof. We note that this result  was the inspiration for a recent paper, \cite{EXTENSIONJ}, in which the number of triangles in a similar graph -- the graph with the same vertex set, but defined by two vertices being adjacent if their difference is a quartic residue rather than a quadratic residue -- was counted.

%---------

\begin{proof}[Proof of Corollary~\ref{thm:littleguy}(\ref{ftwo})]
We note that that Theorem~\ref{thm:selfcomp} implies that $\binom{q}{3}=|\wt{K_3}(G)|+|\wt{P_3}(G)|$ (since there are only $4$ non-isomorphic $3$-vertex graphs), and that a combination of this with (\ref{eq:tildesize}) and (\ref{triangtwo}) gives
\begin{equation}
2|P_3(G)|=\binom{q}{3}-2|K_3(G)|=(1/8) q (q-1)^2. \label{gnom}
\end{equation}
A similar argument to that used to find the value of $\l_1$ from (\ref{triangtwo}) can now be used to obtain $\l_2$ from (\ref{gnom}). 
\end{proof}•

%---------

\begin{proof}[Proof of Corollary~\ref{thm:littleguy}(\ref{fthree})]
We will, as in these previous two proofs, appeal to (\ref{eq:lambda}) to find the value of $\l_3$. We will not, however, count $|D(G)|$, $|R(G)|$, or $|P_4(G)|$ directly as we did in previous proofs; instead, we will only count linear combinations of these parameters, and this will be sufficient for our purposes. To that end, fix $\{x,y\}\in E$ and consider distinct $\{x,y,z\}\in K_3(G)$ and $\{x,y,w\}\in P_3(G)$. Either $\{x,y,z,w\}\in D(G,\{x,y\})$ or $\{x,y,z,w\}\in R(G,\{x,y\})$, and in both cases, $|N_{G[\{x,y,z,w\}]}(x)|\neq |N_{G[\{x,y,z,w\}]}(y)|$. Thus, if we let $D^{1}(G,\{x,y\})$ be all $\{x,y,z,w\}\in D(G,\{x,y\})$ satisfying $|N_{G[\{x,y,z,w\}]}(x)|\neq |N_{G[\{x,y,z,w\}]}(y)|$, and $R^1(G,\{x,y\})$ be defined as the set all  $\{x,y,z,w\} \in R(G,\{x,y\})$ with $\{|N_{G[\{x,y,z,w\}]}(x)|, |N_{G[\{x,y,z,w\}]}(y)|\}=\{2,3\}$, then 
\begin{equation}
|K_3(G,\{x,y\})||P_3(G,\{x,y\})| =|D^{1}(G,\{x,y\})|+|R^{1}(G,\{x,y\})|. \label{uone}
\end{equation}
Now, by a similar argument to that used to develop (\ref{triang}), we can see that 
\begin{equation}
2|P_3(G)|=|E||P_3(G,\{x,y\})|, \label{pathnonsense}
\end{equation}
and a combination of (\ref{triang}) and (\ref{pathnonsense}) with (\ref{triangtwo}), (\ref{gnom}), and (\ref{edges}) tells us that 
\begin{equation}
|E||K_3(G,\{x,y\})||P_3(G,\{x,y\})| = (1/32) q (q-1)^2 (q-5). \label{mindis}
\end{equation}
On the other hand, since there are exactly four edges in $D$ with the property that their endpoints have different degrees, we have that
\begin{equation}
|E||D^1(G,\{x,y\})|=4|D(G)|,
\end{equation}
and by similar reasoning, that
\begin{equation}
|E||R^1(G,\{x,y\})|=2|R(G)|.\label{breakingsobadly}
\end{equation} 
It follows from (\ref{uone}) and (\ref{mindis})-(\ref{breakingsobadly}) that
\begin{equation}
4|D(G)|+2|R(G)| = (1/32) q (q-1)^2 (q-5). \label{finale}
\end{equation}

Following similar steps to those used to get from (\ref{uone}) to (\ref{finale}), we can show that 
\begin{equation}
2|R(G)|+2|P_4(G)|=(1/32) q (q-1)^3. \label{finaletwo}
\end{equation}
It follows from (\ref{finale}) and (\ref{finaletwo}) that 
\begin{align}
(1/32) (q-1)^2 q (q-3)   &=2(|D(G)|+|R(G)|)+|P_4(G)|\\
   &=|\wt{D}(G)|+|\wt{R}(G)|+|\wt{P_4}(G)|,\label{hyper}
\end{align}
with (\ref{hyper}) following from the fact that $G$ is self-complementary. It follows from (\ref{hyper}) and by (\ref{eq:lambda}) that $\l_3$ has the claimed value.
\end{proof}•

%---------

\begin{proof}[Proof of Corollary~\ref{thm:littleguy}(\ref{ffour})]
We notice that, since all $11$ non-isomorphic $4$-vertex graphs are elements of $P_4(G)\cup\wt{C_4}(G)\cup\wt{K_{1,3}}(G)\cup\wt{R}(G)\cup\wt{D}(G)\cup\wt{K_{4}}(G)$, and since the sets in this union are disjoint by definition, 
\begin{equation}
\binom{q}{4}=|P_4(G)|+|\wt{C_4}(G)|+|\wt{K_{1,3}}(G)|+|\wt{R}(G)|+|\wt{D}(G)|+|\wt{K_{4}}(G)|. \label{fourpart}
\end{equation}
It follows from (\ref{hyper}) and (\ref{fourpart})   that 
\begin{equation}
|\wt{C_4}(G)|+|\wt{K_{1,3}}(G)|+|\wt{K_{4}}(G)|=(1/96) (q-5) (q-3) (q-1) q.
\end{equation}
 It follows from (\ref{hyper}) and (\ref{eq:lambda}) that $\l_4$ has the claimed value. 
\end{proof}•

%---------

This completes the proof of all parts of Corollary~\ref{thm:littleguy}. Before moving to the proof of Theorem~\ref{almostfour}, we prove the following lemma from which the theorem will follow almost immediately.

\begin{lemma}\label{almostfourlemma}
We have the following relationships between the parameters of $G$:
\begin{enumerate}[(a)]
\item $4|D(G)|+2|R(G)| = (1/32) q (q-1)^2 (q-5)$\label{ltone}; 
\item $2|R(G)|+|P_4(G)|=(1/32) q (q-1)^3$\label{lttwo};
\item $6k_4+|D(G)| = (1/128) q (q-1) (q-5) (q-9)$\label{ltthree};
\item $|D(G)|+2|C_4(G)|=(1/128) q (q-1)^2  (q-5)$\label{ltfour};
\item $3|K_{1,3}(G)|+|R(G)| = (1/64) q (q-1)^2 (q-5)$\label{ltfive}.
\end{enumerate}
\end{lemma}•

\begin{proof}[Proof of Lemma~\ref{almostfourlemma}]
Since (\ref{ltone}) and (\ref{lttwo}) are (\ref{finale}) and (\ref{finaletwo}), respectively, we begin by showing (\ref{ltthree}). To that end, consider distinct $\{x,y,z\},\{x,y,w\}\in K_3(G)$. Either $\{x,y,z,w\}\in D(G,\{x,y\})$ or $\{x,y,z,w\}\in K_4(G,\{x,y\})$, and in both cases, $|N_{G[\{x,y,z,w\}]}(x)|=|N_{G[\{x,y,z,w\}]}(x)|=3$. Thus, if we let $D^{1}(G,\{x,y\})$ be all elements $\{x,y,z,w\}\in D(G,\{x,y\})$ satisfying $|N_{G[\{x,y,z,w\}]}(x)|=|N_{G[\{x,y,z,w\}]}(x)|=3$, then 
\begin{equation}
\binom{|K_3(G,\{x,y\})|}{2}=|D^{1}(G,\{x,y\})|+|K_4(G,\{x,y\})|. \label{guone}
\end{equation}
Now, we can use a similar argument to that used to trace from  (\ref{uone}) to (\ref{finale}) to obtain from (\ref{guone}) that 
\begin{equation}
6k_4+|D(G)| = (1/128) q (q-1) (q-5) (q-9), \label{lefty}
\end{equation}
which is exactly (\ref{ltthree}). 

A combination of (\ref{eq:transmain}) and the fact that $G$ is self-complementary by Theorem~\ref{thm:selfcomp} implies that 
\begin{align}
|\cl{P_3}(G,\{x,y\})|&=(1/|E|) |\cl{P_3}(G)|=(1/|E|) |P_3(G)|. 
\end{align}
A value for this can be obtained from (\ref{edges}) and (\ref{gnom}), and then a similar argument to that used to show (\ref{ltone}), this time centered around $\binom{|\cl{P_3}(G,\{x,y\})|}2$ rather than $\binom{|K_3(G,\{x,y\})|}{2}$, can show
\begin{equation}
|\cl{D}(G)|+2|\cl{C_4}(G)|=(1/128) q (q-1)^2  (q-5).\label{staringintothesun}
\end{equation}
Then, using the self-complementary of property of $G$ guaranteed by Theorem~\ref{thm:selfcomp}, this implies that
\begin{equation}
|D(G)|+2|C_4(G)|=(1/128) q (q-1)^2  (q-5), \label{righty}
\end{equation}
which is exactly (\ref{ltfour}). A similar proof with $|K_3(G,\{x,y\})||\cl{P_3}(G,\{x,y\})|$ in place of $\binom{|\cl{P_3}(G,\{x,y\})|}2$ (including a similar complementation step) proves (\ref{ltfive}).

\end{proof}

We now proceed to the proof of Theorem~\ref{almostfour}, which will be a short argument now that  Lemma~\ref{almostfourlemma} is in place. 

\begin{proof}[Proof of Theorem~\ref{almostfour}]
It is not difficult to check with standard linear algebra techniques that the five equations presented in Lemma~\ref{almostfourlemma} are linearly independent with respect to the six graph parameters, and thus that we can determine the values of $|P_4(G)|$, $|C_4(G)|$, $|K_{1,3}(G)|$, $|R(G)|$, and  $|D(G)|$ all in terms of $q$ and $|K_4(G)|$. By then recalling that $|H(G)|=|\cl{H}(G)|$ by the self-complementary property of $G$ guaranteed by Theorem~\ref{thm:selfcomp}, and by using (\ref{eq:tildesize}), we have the desired results. 
\end{proof}

\section{On the value of $k_4$}
\label{sec:kfour}

As discussed in Section~\ref{sec:introduction}, if we can determine the value of $k_4$, then we will immediately have new constructions for, and the values for the parameters of, $62$ sequences of designs, with forms similar to those presented in Corollary~\ref{thm:littleguy}. Moreover, by looking at the differences in sizes given by Theorem~\ref{almostfour}, we can see that many of these designs cannot have the same number of blocks, and therefore are guaranteed to be non-isomorphic. In addition to providing designs, determining $k_4$ would provide a nice contribution to the work on the general clique structure of Paley graphs that has been done in \cite{BOLRG,CLIQUEA,CLIQUEB,CLIQUEC,CLIQUED,CLIQUEE,CLIQUEF,CLIQUEG,CLIQUEH,CLIQUEI} and by others, and would provide an extension of the results discussed immediately after the proof of Corollary~\ref{thm:littleguy}(\ref{fone}). Thus, it is natural to pose the following problem. 

\begin{problem}\label{mainproblem}
Determine the value of $k_4$. 
\end{problem}•

At the beginning of this paper we fixed $p,q,r,$ and $G$, and until now all of our results and discussions have been valid whether $G$ is the Paley graph or Peisert graph, as shown in Section~\ref{PROOFS}. However, even though all our results in the preceding sections involving  $k_4$ are the same in both cases, it turns out that $k_4$ need not be the same in both cases, and thus the solution to Problem~\ref{mainproblem} depends on which graph $G$ was fixed to be at the start of this paper (if $p\equiv 3$ and $r$ is even, for in the other case there is no ambiguity). One  can check with a computer (for example, in a way explained in Section~\ref{sec:closing}, Closing Remark $3$), thatif $G$ is assumed to be the Paley graph on $q=7^2$ vertices, then we have $k_4=2450$, while if $G$ is assumed to be the Peisert graph on $q=7^2$ vertices, then we have $k_4=2156$, with this being a minimal counterexample. We keep this in mind as we discuss Problem~\ref{mainproblem}.

Before briefly examining $k_4$ from a more general point of view, let us tabulate the values of $k_4$ for the first $24$ possible values of $q$ if $G$ is assumed to be the Paley graph, and the first $15$ values of $q$ if $G$ is assumed to be the Peisert graph. For an explanation of how the values in these tables were calculated, see Section~\ref{sec:closing}, Closing Remark $3$. Our values for the Paley graph are:
\begin{figure}[h!]
    \begin{tabular}{| l | l | l | l |}
    \hline
    $\b{q}$ & $\b{k_4}$ \\ \hline
    $5$ & $0$ \\ \hline
    $9$ & $0$ \\ \hline
    $13$ & $0$ \\ \hline
    $17$ & $0$ \\ \hline
    $25$ & $75$ \\ \hline
    $29$ &  $203$ \\ \hline
    \end{tabular}
    \quad
    \begin{tabular}{| l | l | l | l |}
    \hline
    $\b{q}$ & $\b{k_4}$ \\ \hline   
    $37$ & $555$ \\ \hline
    $49$ & $2450$ \\ \hline
    $53$ & $3445$ \\ \hline
    $61$ & $6100$ \\ \hline
    $73$ & $13140$ \\ \hline
    $81$ & $21870$ \\ \hline
    \end{tabular}
    \quad
    \begin{tabular}{| l | l | l | l |}
    \hline        
    $\b{q}$ & $\b{k_4}$ \\ \hline  
    $89$ & $31328$ \\ \hline
    $97$ & $46560$ \\ \hline
    $113$ & $87010$ \\ \hline
    $121$ & $118580$ \\ \hline
    $137$ & $197965$ \\ \hline
    $149$ & $275650$ \\ \hline
    \end{tabular}
        \quad
    \begin{tabular}{| l | l | l | l |}
    \hline        
    $\b{q}$ & $\b{k_4}$ \\ \hline  
    $157$ & $346970$ \\ \hline
    $173$ & $520730$ \\ \hline
    $181$ & $619020$ \\ \hline
    $193$ & $802880$ \\ \hline
    $197$ & $868770$ \\ \hline
    $229$ & $1644678$ \\ \hline
    \end{tabular}
            \quad
    \begin{tabular}{| l | l | l | l |}
    \hline        
    $\b{q}$ & $\b{k_4}$ \\ \hline  
    $223$ & $1756820$ \\ \hline
    $241$ & $2024400$ \\ \hline
    $257$ & $2590560$ \\ \hline
    $269$ & $3154025$ \\ \hline
    $277$ & $3535905$ \\ \hline
    $281$ & $3737300$ \\ \hline
    \end{tabular}
\caption{Values of $k_4$ when $G$ is the Paley graph.}\label{fig:pal}
\end{figure}•

Our values for the Peisert graph are:
\begin{figure}[h!]
    \begin{tabular}{| l | l | l | l |}
    \hline
    $\b{q}$ & $\b{k_4}$ \\ \hline
    $9$ & $0$ \\ \hline
    $49$ & $2156$ \\ \hline
    $81$ & $21060$ \\ \hline
    $121$ & $116160$ \\ \hline
    $361$ & $10515930$ \\ \hline
    \end{tabular}
    \,\quad\,
    \begin{tabular}{| l | l | l | l |}
    \hline
    $\b{q}$ & $\b{k_4}$ \\ \hline
    $529$ & $49135636$ \\ \hline
    $729$ & $177655842$ \\ \hline
    $961$ & $544348840$ \\ \hline
    $1849$ & $7534948652$ \\ \hline
    $2209$ & $15347168876$ \\ \hline
    \end{tabular}
        \,\quad\,
    \begin{tabular}{| l | l | l | l |}
    \hline
    $\b{q}$ & $\b{k_4}$ \\ \hline
    $2401$ & $21416199700$ \\ \hline
    $3481$ & $95096986470$ \\ \hline
    $4489$ & $263148591640$ \\ \hline
    $5041$ & $418750123260$ \\ \hline
    $6241$ & $984209819060$ \\ \hline
    \end{tabular}
    \caption{Values of $k_4$ when $G$ is the Peisert graph.}\label{fig:peis}
\end{figure}

Let us also get a feel for the types of designs that we will obtain once the value of $k_4$ is known. For this, assume that $q=29$, in which case $G$ is the Paley graph and $k_4=203$. Then, direct applications of Theorem~\ref{thm:backbone} and Theorem~\ref{almostfour} with this value of $k_4$ in place give us $64$ designs ($62$ of which are new, two of which are given by Corollary~\ref{thm:littleguy}(\ref{fthree}) and Corollary~\ref{thm:littleguy}(\ref{fthree})) having parameters as specified in the table in Section~\ref{sec:appendix}.

Now, let us examine $k_4$ a bit further. Since $|H(G)|$ is a polynomial in $q$ when $H$ is a graph on fewer than $4$ vertices (as shown in Section~\ref{PROOFS}), it is natural to think, a priori, that $k_4$ might be a polynomial in $q$ as well. However, if we note that $k_4$ would have to be a polynomial of degree at most $4$, since $k_4\leq\binom{q}{4}$ by definition, and if we then try to interpolate values from Figure~\ref{fig:pal} or from Figure~\ref{fig:peis}, we realize that $k_4$ cannot be a polynomial. In fact, even if we only consider $k_4$ over $r=1$ (i.e., over the Paley graphs of prime order), or try to consider $k_4$ over $p\equiv 1$ (mod $8$), or other similar subsequences of values, we find that over none of these subsequences can $k_4$ be a polynomial.  

It is worth noting that the techniques used in Section~\ref{PROOFS} provide no clear method for counting $k_4$. With this and the previous paragraph in mind, it is natural to look to other previously developed techniques to solve Problem~\ref{mainproblem}. Recalling the comments immediately following the proof of Corollary~\ref{thm:littleguy}(\ref{fone}), one might want to try adapting the techniques used  in \cite{CLIQUEI} and \cite{EXTENSIONJ} to count $3$-vertex cliques in the Paley graph. We note that the techniques of \cite{CLIQUEI} and \cite{EXTENSIONJ} center around the fact that every $2$-vertex clique lies in the same number of $3$-cliques in the Paley graph, and that, as one can verify (for example, by computer), not every $3$-vertex clique lies in the same number of $4$-vertex cliques in $G$ (even when $G$ is a Paley graph), with a minimal counterexample occurring when $q=p=37$; thus, these techniques will fail as well. Problem~\ref{mainproblem}, which remains open, may require some very different techniques than those which are commonly used in this area. We do show in the following section, though, that we can estimate the value of $k_4$ quite closely if $q$ is large and $G$ is assumed to be the Paley graph. 

It is worth mentioning, perhaps, that a full classification of the maximal cliques in $G$ (regardless of whether $G$ is the Paley or Peisert graph) would enable computation of $k_4$ using the Inclusion-Exclusion Principle. However, such classifications are themselves difficult, unsolved problems. 

\section{On the value of $k_4$ when $G$ is the Paley graph}
\label{sec:kfp}

Assume, for the duration of this section, that $G$ was chosen to be the Paley graph. We will begin this section by finding estimates and precise asymptotics for $k_4$ in this case. After estimates and asymptotics have been provided, we will derive an exact expression for $k_4$ in terms of quadratic residue characters (which we will provide a definition of ), but we will explain that this sum is likely very difficult to simplify using known methods. To begin, we turn our attention to the following result, which was proven by Thomason in \cite{THOMSP}; a concise proof of this result, as well as some discussion and related observations, is presented on Page $363$ of \cite{BOLRG}. For the duration of this section, as in Section~\ref{PROOFS}, for any graph $H$, let $E(H)$ denote the set of edges of $H$. 

\begin{theorem}[Thomason, \cite{THOMSP}]\label{thm:savior}
For any $U\sub V$, 
\begin{equation}
\abs{ |E(G[U])| - \frac{1}{2}\binom{|U|}{2}}\leq\frac{|U|(q-|U|)}{4\sqrt{q}}. 
\end{equation}•
\end{theorem}•

The following theorem, which precisely determines the behavior of $k_4$ for large $q$, now follows with a few short arguments. 

\begin{theorem}
We have that
\begin{equation}
k_4=q^4/1536+O(q^{7/2}), \label{crywithme}
\end{equation}•
and so, as $q\to\infty$, 
\begin{equation}
k_4\sim q^4/1536.\label{andagainplease}
\end{equation}•
\end{theorem}•

\begin{proof}
Fix any $\{x,y\}\in E$ and let $N_{xy}:=N_G(x)\cap N_G(y)$. It follows from (\ref{eq:transmain}) that
\begin{equation}
6k_4=|E||K_4(G,\{x,y\})|\label{forkinther}.
\end{equation}
Moreover, for $z,w\in V\sm\{x,y\}$, we know that $\{x,y,z,w\}\in K_4(G,\{x,y\})$ if and only if $z,w\in N_{xy}$ and $\{z,w\}\in E$, and thus that
\begin{equation}
|K_4(G,\{x,y\})|=|E(G[N_{xy}])|. \label{crawlingalong}
\end{equation}•
Now, by Lemma~\ref{thm:srg}, we have that 
\begin{equation}
|N_{xy}|=(1/4) (q-5).\label{mediumseefefs}
\end{equation}
It follows from Theorem~\ref{thm:savior}, (\ref{crawlingalong}), and (\ref{mediumseefefs}) that
\begin{equation}
\abs{ |K_4(G,\{x,y\})| - (1/64) (q-5)(q-9)}\leq ((q-5) (3 q+5))/(64 \sqrt{q}), 
\end{equation}•
and in particular, that 
\begin{equation}
|K_4(G,\{x,y\})|=(1/64) q^2+O(q^{3/2}).\label{forkinthertwo}
\end{equation}•
Our main claim, (\ref{crywithme}), now follows from a combination of (\ref{forkinther}) and (\ref{forkinthertwo}) together with (\ref{edges}); from (\ref{crywithme}), we immediately obtain (\ref{andagainplease}). 
\end{proof}•

We will now move to providing an exact expression for $k_4$ in terms of quadratic characters. As is standard, define the quadratic residue character on $\t{GF}(q)$ by, for any $x\in \t{GF}(q)$, the relation  $\chi(x):=x^{(q-1)/2}$ for $x\neq 0$ and $\chi(0)=0$; we notice that, for $x\neq 0$, we have $\chi(x)=1$ if and only if $x$ is a quadratic residue in $\t{GF}(q)$, and that $\chi(x)=-1$ otherwise. For an introduction to quadratic residue characters, see Sections $13.1-13.2$ of \cite{BOLRG}. Combining ideas used by Andrews to estimate the number of triples of consecutive quadratic residues in $\t{GF}(p)$ for prime $p$ in Section $10.2$ of \cite{ANDREWS} with some properties which we have observed about the Paley graph, we prove the following:

\begin{theorem}
Taking the following sum over distinct $a,b\in\t{GF}(q)\sm\{0,1\}$:
\begin{equation}
k_4=(1/512)q(q-1)\sum\pr{1+\chi(a-b)}\prod_{i\in\{0,1\}}\pr{1+\chi(a-i)}\pr{1+\chi(b-i)}.\label{pickuptheradio}
\end{equation}•
\end{theorem}•

\begin{proof}
In light of (\ref{forkinther}), we need only concern ourselves with determining $|K_4(G,\{x,y\})|$ for a given $\{x,y\}\in E$, and we will focus on determining $|K_4(G,\{0,1\})|$ for ease of notation. Now, we know that $\{0,1,x,y\}\in K_4(G,\{0,1\})$ if and only if $\chi(x)=\chi(y)=\chi(x-1)=\chi(y-1)=\chi(x-y)=1$ for all distinct $x,y\in V\sm\{0,1\}$; it follows that if
\begin{equation}
f(x,y):=\pr{1+\chi(x)}\pr{1+\chi(x-1)}\pr{1+\chi(y)}\pr{1+\chi(x-1)}\pr{1+\chi(x-y)}, 
\end{equation}•
then $f(x,y)=32$ if $\{0,1,x,y\}\in K_4(G,\{0,1\})$ and $f(x,y)=0$ otherwise. Therefore, with the following sum taken over distinct $x,y\in V\sm\{0,1\}$:
\begin{equation}
|K_4(G,\{0,1\})|=\sum(1/32)f(x,y).\label{pigsonthewing}
\end{equation}•
It follows from (\ref{pigsonthewing}), (\ref{forkinther}), and (\ref{edges}) that (\ref{pickuptheradio}) holds. 
\end{proof}•

Though some small simplifications can be made, sums of characters such as those in (\ref{pickuptheradio}) are notoriously difficult to simplify in any meaningful way, or to even obtain bounds from (see, for example, the discussions in Section $10.2$ of \cite{ANDREWS} for more on this point). For this reason, we leave this theorem as is and rely only on the methods presented earlier in this section for estimates.

\section{On the number of designs given by Theorem~\ref{thm:backbone}}
\label{sec:bigk}

Since introducing Theorem~\ref{thm:backbone} in Section~\ref{sec:introduction}, our focus has been on using this theorem to construct $2$-$(q,k,\l)$-designs, where $k\in\{3,4\}$ from the induced subgraph structure of $G$. Of course, this only scratches the surface, since Theorem~\ref{thm:backbone} clearly provides many designs of different sizes. To explore this point, let $\D=\D(q)$ denote the number of unique design constructions admitted by Theorem~\ref{thm:backbone}, and for fixed $k\in\N$, let $\d_k$ denote the number of constructions of designs specified by Theorem~\ref{thm:backbone} with parameters of the form $(q,k,\l)$ for some $\l$; also, let $\gamma_k$ denote the number of non-isomorphic graphs on $k$ vertices. By definition of $\gamma_k$, we can find distinct $k$-vertex graphs $H_1,\ld,H_{\gamma_k/2}$ so that $H_i\neq\cl{H_j}$ for all distinct $i,j\in[\gamma_k/2]$. By considering combinations of these graphs, it is immediate that $\d_k\geq 2^{\gamma_k/2}$. Moreover, equality would only hold if all $\gamma_k$ graphs on $k$ vertices were non-self-complementary, which surely would not be the case if $k> 3$, and so it follows that 
\begin{equation}
\D>\sum_{k=3}^q   2^{\gamma_k/2}.
\end{equation}
Over $k\in\N$, $\gamma_k$ has a rich history that we cannot hope to reasonably discuss here, so we refer the interested reader to II.2 in \cite{NONISOL}, and simply note that 
\begin{equation}
\gamma_k>2^{\binom{k}{2}}/k!. 
\end{equation}

While $2^{\gamma_k/2}$ unique design constructions for each $k\in\{3,\ld,q\}$ is certainly a very large amount, to really understand the strength of this statement, we would need to estimate how many of these designs are non-isomorphic. If most will be non-isomorphic, as one might expect from our discussion about $k\in\{3,4\}$, then surely the motivation to study this theorem would be much greater. Thus, it is natural to pose the following which, in this sense, measures the strength of Theorem~\ref{thm:backbone}. 
\begin{problem}
Fix $k\in\{3,\ld,q\}$. How many pairs of distinct sets of $k$-vertex graphs, $(\{H_1,\ld,H_m\},\{I_1,\ld,I_n\})$ for some $n,m\in\N$, are there satisfying $H_i\notin\{H_j,\cl{H_j}\}$ and $I_k\notin\{I_l,\cl{I_l}\}$ over distinct $i,j\in[n]$ and $k,l\in[m]$, and also satisfying 
\begin{equation}
\sum_{i=1}^m|\wt{H_i}(G)|= \sum_{k=1}^n|\wt{I_k}(G)|?
\end{equation}
\end{problem}•

%\begin{problem}
%Fix $k\in\{3,\ld,q\}$. How many pairs $(\m{H},\m{I})$ with $\m{H}\neq\m{I}$ are there which satisfy, for some $m,n\in\N$, that $\m{H}=\{H_1,\ld,H_m\}$, $\m{I}=\{I_1,\ld,I_n\}$, that $H_1,\ld,H_m,I_1,\ld,I_n$ are $k$-vertex graphs, $H_i\neq\cl{H_j}$ for all $i,j\in[n]$, $I_{k}\neq\cl{I_l}$ for all all $k,l\in[n]$, and that 
%\begin{equation}
%\sum_{i=1}^m|\wt{H_j}(G)|= \sum_{k=1}^n|\wt{I_k}(G)|?
%\end{equation}
%\end{problem}

\section{Closing remarks}
\label{sec:closing}
\begin{enumerate}
\item It is possible to present the results of this paper while completely avoiding graph-theoretic language.  For example, in the case that $G$ is a Paley Graph, if we let $\square_q$ denote the set of quadratic residues in $\t{GF}(q)$, then Corollary~\ref{thm:littleguy}(\ref{fone}) can be restated as follows: Let $V:=\t{GF}(q)$ and let $\m{B}_1$ be the set of all $S\sub \t{GF}(q)$ such that $|S|=3$ and either $x-y\in\sq$ for all $x,y\in S$, or $x-y\notin\sq$ for all $x,y\in S$. Then, $(V,\m{B}_1)$ forms a $2$-$(p,3,\l_1)$-design, where $\l_1=(1/4) (p-5)$. Moreover, it may be particularly tempting to some to try to remove all graph-theoretic language this way, and to try to present the results of this paper entirely as those concerned with using symmetries (group actions) to construct designs, since this is one of the few standard methods for design construction. However, we invite those tempted to try to remove all graph theoretic language from Corollary~\ref{thm:littleguy}(\ref{fthree}) and Corollary~\ref{thm:littleguy}(\ref{ffour}) as we just did with Corollary~\ref{thm:littleguy}(\ref{fone}), and then to remove the graph-theoretic language from their proofs; the notation inevitably becomes cumbersome to the point of illegibility. \\

\item For convenience, we note that the complementary designs of those presented in Corollary~\ref{thm:littleguy}, which we denote with an overbar, have the following parameters:
\begin{enumerate}
\item $\cl{(V,\m{B}_1)}$ forms a $2$-$(q,q-3,\l_1)$-design with $\l_1=(1/12) (q-3) (q-4) (q-5)$;
\item $\cl{(V,\m{B}_2)}$ forms a $2$-$(q,q-3,\l_2)$-design with $\l_2=(1/8) (q-1) (q-3) (q-4) $;
\item $\cl{(V,\m{B}_3)}$ forms a $2$-$(q,q-4,\l_3)$-design with $\l_3=(1/32) (q-1) (q-3) (q-4) (q-5)$;
\item $\cl{(V,\m{B}_4)}$ forms a $2$-$(q,q-4,\l_4)$-design with $\l_4= (1/96) (q-3) (q-4) (q-5)^2$.
\end{enumerate}\textbf{}\\

\item \label{threetre} The tables of $k_4$ values provided in Section~\ref{sec:kfour} were generated using Magma. One can use more common languages like Sage to generate such values over the Paley graph quite easily, and not much is lost; however, for both computation time purposes and because certain algebraic elements can be tricky in Sage, we  recommend using Magma or another more algebraically-oriented language for Peisert graph calculations. For convenience, we will provide Sage code for Paley calculations and Magma code for Peisert calculations. In both cases, to improve our computation times significantly, we will appeal to (\ref{forkinther}). To generate $k_4$ values for the Paley graph in Sage, we can proceed as follows:
\vspace{0.1in}
\begin{center}
\begin{verbatim}
from itertools import product  #Allows for Cartesian products 
F.<a> = GF(q)  #F is finite field GF(q)
G = Graph([F, lambda i,j: i!=j and (i-j).is_square()])#Paley graph
G.relabel() #For convenience
V = G.vertices() #For convenience
E = G.edges(labels = false) #For convenience
def k4(G): #Determines the values of k_4
    count = 0
    for (x,y) in filter(lambda (x,y): (1<x<y), product(V,V)):
        if (0,x) in E and (0,y) in E and (1,x) in E and (1,y) in E: 
            if (x,y) in E:
                count += 1
    return (1/24)*q*(q-1)*count
\end{verbatim}
\end{center}
\vspace{0.1in}
\normalsize The return statement in k4 is determined by (\ref{forkinther}) and (\ref{edges}). From this, we have the table provided in Section~\ref{sec:kfour} for the Paley graph. The table of designs in Section~\ref{sec:appendix} can be generated using the $k_4$ value obtained from this when $q=29$ together with Theorem~\ref{almostfour} and Theorem~\ref{thm:backbone}, as explained further in Section~\ref{sec:appendix}. To generate $k_4$ values for the Peisert graph in Magma, let us not define the graph or any functions, but rather just proceed in a more efficient way by: 
\vspace{0.1in}
\begin{center}
\begin{verbatim}
F<a> := FiniteField(p,r);  //q = p^r and F = GF(q)
N0 := {};  //This is N_G(0)
for j in [0..p^r] do
    N0 := N0 join {a^(4*j), a^(4*j+1)};
end for;		
N1 := {};  //This is N_G(1)
for j in [0..p^r] do
    N1 := N1 join {a^(4*j)+1, a^(4*j+1)+1};
end for;		
N :=  (N0 meet N1) diff {0,1};  //This is N_G(0)\cap N_G(1)
count := 0;
for S in Subsets(N, 2) do //Computes |K_4(G,{0,1})|
    for x in S do
        for y in (S diff {x}) do
            if (x-y) in N0 then
                count := count + 1;
            end if;
        end for;
    end for;
end for;
print (1/48)*(p^r)*(p^r-1)*count;
\end{verbatim}
\end{center}
\vspace{0.1in}
\normalsize

\item Though the designs which he obtained differ from those which we obtain here, and though his techniques differ significantly, Tonchev showed that designs can be found in Rank $3$ graphs (a class which contains Paley graphs and Peisert graphs) in \cite{TONCHEV1,TONCHEV2}. 
\end{enumerate}

•

\section{Acknowledgments}
\label{sec:acknowledgments}
I would like to thank Felix Lazebnik, Ron Baker, and Qing Xiang for many helpful suggestions which improved this work very much. I would like to thank Chris Godsil, Robert Coulter, Felix Goldberg, Avi Kulkarni and the CECM gang, and Ron Baker for helpful and enjoyable discussions concerning $k_4$. I would also like to thank linguist Amanda Payne for editing the final draft of this paper. 

•

\section{Appendix}
\label{sec:appendix}

The following table shows the values of $\l$ for the $62$ nontrivial designs which can be constructed from $4$-vertex subgraphs of $G$ using Theorem~\ref{thm:backbone}, if we assume that $q=29$ and use Theorem~\ref{almostfour} and the value $k_4=203$ given in Figure~\ref{fig:pal}. They are all $2$-$(29,4,\l)$-designs with block set $\m{B}$, where $\l$ and $\m{B}$ are as specified in the table. 
\vspace{-0.05in}
\begin{figure}
        \setlength\extrarowheight{5pt}
    \begin{tabular}{| l | l |}
    \hline
    $\boldsymbol\l$ &  $\bm{\m{B}}$ \\ \hline

     $6$ &  $\wt{K_4}(G)$ \\ \hline
     $54$ &  $\wt{D}(G)$ \\ \hline
     $144$ &  $\wt{R}(G)$ \\ \hline
     $90$ &  $\wt{C_4}(G)$ \\ \hline
     $36$ &  $\wt{K_{1,3}}(G)$ \\ \hline
     $150$ &  $\wt{P_4}(G)$ \\ \hline

     $60$ &  $\wt{K_4}(G)\cup\wt{D}(G)$ \\ \hline
     $150$ &  $\wt{K_4}(G)\cup\wt{R}(G)$ \\ \hline
     $96$ &  $\wt{K_4}(G)\cup\wt{C_4}(G)$ \\ \hline
     $42$ &  $\wt{K_4}(G)\cup\wt{K_{1,3}}(G)$ \\ \hline
     $156$ &  $\wt{K_4}(G)\cup\wt{P_4}(G)$ \\ \hline
     $198$ &  $\wt{D}(G)\cup\wt{R}(G)$ \\ \hline
     $144$ &  $\wt{D}(G)\cup\wt{C_4}(G)$ \\ \hline
     $90$ &  $\wt{D}(G)\cup\wt{K_{1,3}}(G)$ \\ \hline
     $204$ &  $\wt{D}(G)\cup\wt{P_4}(G)$ \\ \hline
     $234$ &  $\wt{R}(G)\cup\wt{C_4}(G)$ \\ \hline
     $180$ &  $\wt{R}(G)\cup\wt{K_{1,3}}(G)$ \\ \hline
     $294$ &  $\wt{R}(G)\cup\wt{P_4}(G)$ \\ \hline
     $126$ &  $\wt{C_4}(G)\cup\wt{K_{1,3}}(G)$ \\ \hline
     $240$ &  $\wt{C_4}(G)\cup\wt{P_4}(G)$ \\ \hline
     $186$ &  $\wt{K_{1,3}}(G)\cup\wt{P_4}(G)$ \\ \hline

     $204$ &  $\wt{K_4}(G)\cup\wt{D}(G)\cup\wt{R}(G)$ \\ \hline
     $150$ &  $\wt{K_4}(G)\cup\wt{D}(G)\cup\wt{C_4}(G)$ \\ \hline
     $96$ &  $\wt{K_4}(G)\cup\wt{D}(G)\cup\wt{K_{1,3}}(G)$ \\ \hline
     $210$ &  $\wt{K_4}(G)\cup\wt{D}(G)\cup\wt{P_4}(G)$ \\ \hline
     $240$ &  $\wt{K_4}(G)\cup\wt{R}(G)\cup\wt{C_4}(G)$ \\ \hline
     $186$ &  $\wt{K_4}(G)\cup\wt{R}(G)\cup\wt{K_{1,3}}(G)$ \\ \hline
     $300$ &  $\wt{K_4}(G)\cup\wt{R}(G)\cup\wt{P_4}(G)$ \\ \hline
     $132$ &  $\wt{K_4}(G)\cup\wt{C_4}(G)\cup\wt{K_{1,3}}(G)$ \\ \hline
     $246$ &  $\wt{K_4}(G)\cup\wt{C_4}(G)\cup\wt{P_4}(G)$ \\ \hline
     $192$ &  $\wt{K_4}(G)\cup\wt{K_{1,3}}(G)\cup\wt{P_4}(G)$ \\ \hline
    \end{tabular}
    \begin{tabular}{| l | l | l | l |}
    \hline
    $\boldsymbol\l$ &  $\bm{\m{B}}$ \\ \hline

     $288$ &  $\wt{D}(G)\cup\wt{R}(G)\cup\wt{C_4}(G)$ \\ \hline
     $234$ &  $\wt{D}(G)\cup\wt{R}(G)\cup\wt{K_{1,3}}(G)$ \\ \hline
     $348$ &  $\wt{D}(G)\cup\wt{R}(G)\cup\wt{P_4}(G)$ \\ \hline
     $180$ &  $\wt{D}(G)\cup\wt{C_4}(G)\cup\wt{K_{1,3}}(G)$ \\ \hline
     $294$ &  $\wt{D}(G)\cup\wt{C_4}(G)\cup\wt{P_4}(G)$ \\ \hline
     $240$ &  $\wt{D}(G)\cup\wt{K_{1,3}}(G)\cup\wt{P_4}(G)$ \\ \hline
     $270$ &  $\wt{R}(G)\cup\wt{C_4}(G)\cup\wt{K_{1,3}}(G)$ \\ \hline
     $384$ &  $\wt{R}(G)\cup\wt{C_4}(G)\cup\wt{P_4}(G)$ \\ \hline
     $330$ &  $\wt{R}(G)\cup\wt{K_{1,3}}(G)\cup\wt{P_4}(G)$ \\ \hline
     $276$ &  $\wt{C_4}(G)\cup\wt{K_{1,3}}(G)\cup\wt{P_4}(G)$ \\ \hline

     $294$ &  $\wt{K_4}(G)\cup\wt{D}(G)\cup\wt{R}(G)\cup\wt{C_4}(G)$ \\ \hline
     $240$ &  $\wt{K_4}(G)\cup\wt{D}(G)\cup\wt{R}(G)\cup\wt{K_{1,3}}(G)$ \\ \hline
     $354$ &  $\wt{K_4}(G)\cup\wt{D}(G)\cup\wt{R}(G)\cup\wt{P_4}(G)$ \\ \hline
     $186$ &  $\wt{K_4}(G)\cup\wt{D}(G)\cup\wt{C_4}(G)\cup\wt{K_{1,3}}(G)$ \\ \hline
     $300$ &  $\wt{K_4}(G)\cup\wt{D}(G)\cup\wt{C_4}(G)\cup\wt{P_4}(G)$ \\ \hline
     $246$ &  $\wt{K_4}(G)\cup\wt{D}(G)\cup\wt{K_{1,3}}(G)\cup\wt{P_4}(G)$ \\ \hline
     $276$ &  $\wt{K_4}(G)\cup\wt{R}(G)\cup\wt{C_4}(G)\cup\wt{K_{1,3}}(G)$ \\ \hline
     $390$ &  $\wt{K_4}(G)\cup\wt{R}(G)\cup\wt{C_4}(G)\cup\wt{P_4}(G)$ \\ \hline
     $336$ &  $\wt{K_4}(G)\cup\wt{R}(G)\cup\wt{K_{1,3}}(G)\cup\wt{P_4}(G)$ \\ \hline
     $282$ &  $\wt{K_4}(G)\cup\wt{C_4}(G)\cup\wt{K_{1,3}}(G)\cup\wt{P_4}(G)$ \\ \hline
     $324$ &  $\wt{D}(G)\cup\wt{R}(G)\cup\wt{C_4}(G)\cup\wt{K_{1,3}}(G)$ \\ \hline
     $438$ &  $\wt{D}(G)\cup\wt{R}(G)\cup\wt{C_4}(G)\cup\wt{P_4}(G)$ \\ \hline
     $384$ &  $\wt{D}(G)\cup\wt{R}(G)\cup\wt{K_{1,3}}(G)\cup\wt{P_4}(G)$ \\ \hline
     $330$ &  $\wt{D}(G)\cup\wt{C_4}(G)\cup\wt{K_{1,3}}(G)\cup\wt{P_4}(G)$ \\ \hline
     $420$ &  $\wt{R}(G)\cup\wt{C_4}(G)\cup\wt{K_{1,3}}(G)\cup\wt{P_4}(G)$ \\ \hline
    
     $330$ &  $\wt{K_4}(G)\cup\wt{D}(G)\cup\wt{R}(G)\cup\wt{C_4}(G)\cup\wt{K_{1,3}}(G)$ \\ \hline 
     $444$ &  $\wt{K_4}(G)\cup\wt{D}(G)\cup\wt{R}(G)\cup\wt{C_4}(G)\cup\wt{P_4}(G)$ \\ \hline
     $390$ &  $\wt{K_4}(G)\cup\wt{D}(G)\cup\wt{R}(G)\cup\wt{K_{1,3}}(G)\cup\wt{P_4}(G)$ \\ \hline 
     $336$ &  $\wt{K_4}(G)\cup\wt{D}(G)\cup\wt{C_4}(G)\cup\wt{K_{1,3}}(G)\cup\wt{P_4}(G)$ \\ \hline
     $426$ &  $\wt{K_4}(G)\cup\wt{R}(G)\cup\wt{C_4}(G)\cup\wt{K_{1,3}}(G)\cup\wt{P_4}(G)$ \\ \hline
     $474$ &  $\wt{D}(G)\cup\wt{R}(G)\cup\wt{C_4}(G)\cup\wt{K_{1,3}}(G)\cup\wt{P_4}(G)$ \\ \hline 
    \end{tabular}
\caption{The $2$-$(29,4,\l)$-designs given by Theorem~\ref{thm:backbone}.}.
\end{figure}
\np

\bibliographystyle{amsplain}
\bibliography{paleybib}

\providecommand{\bysame}{\leavevmode\hbox to3em{\hrulefill}\thinspace}
\providecommand{\MR}{\relax\ifhmode\unskip\space\fi MR }
\providecommand{\MRhref}[2]{%
  \href{http://www.ams.org/mathscinet-getitem?mr=#1}{#2}
}
\providecommand{\href}[2]{#2}
\begin{thebibliography}{10}

\bibitem{COSPB}
A.~Abiad and W.~H. Haemers, \emph{Cospectral graphs and regular orthogonal
  matrices of level 2}, Electron. J. Combin. \textbf{19} (2012), no.~3, Paper
  13, 16. \MR{2967218}

\bibitem{GRAPHDESL}
A.~Ahadi, N.~Besharati, E.~S. Mahmoodian, and M.~Mortezaeefar, \emph{Silver
  block intersection graphs of steiner $2$-designs}, Graphs and Combinatorics
  \textbf{29} (2013), 735 -- 746.

\bibitem{ANDREWS}
G.~Andrews, \emph{Number theory}, Dover Books on Mathematics, Dover
  Publications Inc.; Edition 1, 1971.

\bibitem{GRAPHDESJ}
R.~A. Bailey and Peter~J. Cameron, \emph{Using graphs to find the best block
  designs}, Topics in structural graph theory, Encyclopedia Math. Appl., 147,
  Cambridge Univ. Press, Cambridge, 2013.

\bibitem{CLIQUED}
R.~D. Baker, G.~L. Ebert, J.~Hemmeter, and A.~Woldar, \emph{Maximal cliques in
  the {P}aley graph of square order}, Journal of Statistical Planning and
  Inference \textbf{56} (1996), 33 -- 38.

\bibitem{GRAPHDESG}
E.~J. Billington, \emph{The metamorphosis of $\lambda$-fold $4$-wheel systems
  into $\lambda$-fold bowtie systems}, Australas. J. Combin. \textbf{22}
  (2000), 265 -- 286.

\bibitem{GRAPHDESF}
\bysame, \emph{The extended metamorphosis of a complete bipartite design into a
  cycle system}, Discrete Mathematics \textbf{284} (2004), 63 -- 70.

\bibitem{GRAPHDESH}
E.~J. Billington and C.~C. Lindner, \emph{The metamorphosis of $\lambda$-fold
  $4$-wheel systems into $\lambda$-fold $4$-cycle systems.}, Util. Math.
  \textbf{59} (2001), 215 -- 235.

\bibitem{GRAPHDESE}
E.~J. Billington and G.~Quattrocchi, \emph{The metamorphosis of $\lambda$-fold
  {${K_{3,3}}$}-designs into $\lambda$-fold $6$-cycle systems}, Ars Combin.
  \textbf{64} (2002), 65 -- 80.

\bibitem{COSPA}
Z.~L. Bl{\'a}zsik, J.~Cummings, and W.~H. Haemers, \emph{Cospectral regular
  graphs with and without a perfect matching}, Discrete Math. \textbf{338}
  (2015), no.~3, 199--201. \MR{3291884}

\bibitem{CLIQUEA}
A.~Blokhuis, \emph{On subsets of {${GF(q^2)}$} with square differences},
  Nederl. Akad. Wetensch. Indag. Math. \textbf{46} (1984), 369 -- 372.

\bibitem{CLIQUEB}
A.~Blokhuis, \'{A}. Seress, and H.~A. Wilbrink, \emph{Characterization of
  complete exterior sets of conics}, Combinatorica \textbf{12} (1992), 143 --
  147.

\bibitem{BOLRG}
B.~Bollob\'{a}s, \emph{Random graphs}, Cambridge Studies in Advanced
  Mathematics (Book 73), Cambridge University Press; Edition 2, 2001.

\bibitem{CLIQUEE}
I.~Broere, D.~D{\"o}man, and J.~N. Ridley, \emph{The clique numbers and
  chromatic numbers of certain {P}aley graphs}, Quaestiones Math. \textbf{11}
  (1988), 91 -- 93.

\bibitem{COSPE}
A.~E. Brouwer and E.~Spence, \emph{Cospectral graphs on 12 vertices}, Electron.
  J. Combin. \textbf{16} (2009), no.~1, Note 20, 3. \MR{2515760 (2010f:05113)}

\bibitem{EXTENSIONJ}
M.~Budden, N.~Calkins, W.~N. Hack, J.~Lambert, and K.~Thompson,
  \emph{Enumeration of triangles in quartic residue graphs}, Integers
  \textbf{1} (2012), 57 -- 73.

\bibitem{SIXER}
C.~Y. Chao, \emph{On the classification of symmetric groups with a prime number
  of vertices}, Trans. Amer. Math. Soc. \textbf{158} (1971), 247 -- 256.

\bibitem{CLIQUEC}
S.~D. Cohen, \emph{Characterization of complete exterior sets of conics},
  Quaestiones Math. \textbf{11} (1988), 225 -- 231.

\bibitem{GRAPHDESK}
A.~Devillers, M.~Giudici, C.~H. Li, and C.~E. Praeger, \emph{Locally
  $s$-distance transitive graphs and pairwise transitive designs}, Journal of
  Combinatorial Theory, Series A \textbf{120} (2013), 1855 -- 1870.

\bibitem{PALEYPATCHONE}
P.~Erd{\H{o}}s and A.~R{\'{e}}nyi, \emph{Asymmetric graphs}, Acta Mathematica
  Academiae Scientiarum Hungaricae \textbf{14} (1963), 295 -- 315.

\bibitem{CLIQUEG}
G.~Exoo, S.~Radziszowski, X.~Xiaodong, and X.~Zheng, \emph{Constructive lower
  bounds on classical multicolor {R}amsey numbers}, Electron. J. Combin.
  \textbf{19} (11), Paper 35.

\bibitem{NONISOL}
P.~Flajolet and R.~Sedgewick, \emph{Analytic combinatorics}, Cambridge
  University Press, 2009.

\bibitem{COSPF}
N.~Ghareghani, F.~Ramezani, and B.~Tayfeh-Rezaie, \emph{Graphs cospectral with
  starlike trees}, Linear Algebra Appl. \textbf{429} (2008), no.~11-12,
  2691--2701. \MR{2455524 (2009g:05102)}

\bibitem{GRAPHDESI}
Manjula~C. Gudgeri, H.~G. Shekharappa, and Shailaja~S. Shirkol, \emph{Pbib
  designs and association scheme arising from minimum total dominating sets of
  non square lattice graph}, Gen. Math. Notes \textbf{17} (2013), 91 -- 102.

\bibitem{COSPC}
W.~H. Haemers and F.~Ramezani, \emph{Graphs cospectral with {K}neser graphs},
  Combinatorics and graphs, Contemp. Math., vol. 531, Amer. Math. Soc.,
  Providence, RI, 2010, pp.~159--164. \MR{2757797 (2012f:05176)}

\bibitem{COSPH}
W.~H. Haemers and E.~Spence, \emph{Enumeration of cospectral graphs}, European
  J. Combin. \textbf{25} (2004), no.~2, 199--211. \MR{2070541 (2005d:05102)}

\bibitem{COSPJ}
C.~R. Johnson and M.~Newman, \emph{A note on cospectral graphs}, J. Combin.
  Theory Ser. B \textbf{28} (1980), no.~1, 96--103. \MR{565513 (81d:05051)}

\bibitem{PSEUDOPEIS}
A.~Kisielewicz and W.~Peisert, \emph{Pseudo-random properties of
  self-complementary symmetric graphs}, Journal of Graph Theory \textbf{47}
  (2004), no.~4, 310--316.

\bibitem{GRAPHDESM}
J.~H. Koolen and A.~Munemasa, \emph{Tight $2$-designs and perfect $1$-codes in
  doob graphs}, Journal of Statistical Planning and Inference \textbf{86}
  (2000), 505 -- 513.

\bibitem{PSEUDOKB}
M.~Krivelevich and B.~Sudakov, \emph{Pseudo-random graphs}, Bolyai Society
  Mathematical Studies, vol.~15, Springer Berlin Heidelberg, 2006 (English).

\bibitem{GRAPHDESA}
S.~K{\"u}{\c{c}}{\"u}k{\c{c}}if{\c{c}}i, \emph{Metamorphosis problems for graph
  designs}, Ars Combin. \textbf{113} (2014), 47 -- 64.

\bibitem{COSPD}
M.~Lepovi{\'c}, \emph{Some new results on walk regular graphs which are
  cospectral to its complement}, Discrete Math. \textbf{310} (2010), no.~4,
  767--773. \MR{2574825 (2011b:05287)}

\bibitem{COSPI}
M.~Lepovi{\'c} and I.~Gutman, \emph{No starlike trees are cospectral}, Discrete
  Math. \textbf{242} (2002), no.~1-3, 291--295. \MR{1874772 (2002i:05081)}

\bibitem{GRAPHDESC}
C.~C. Lindner and A.~Lo~Fargo, G.and~Tripodi, \emph{The metamorphosis of
  $\lambda$-fold kite systems into maximum packings of {${\lambda K_n}$} with
  triangles}, J. Combin. Math. Combin. Comput. \textbf{56} (2006)), 171 -- 189.

\bibitem{GRAPHDESD}
C.~C. Lindner and A.~Rosa, \emph{The metamorphosis of $\lambda$-fold block
  designs with block size $4$ into $\l$-fold triple-systems}, J. Statist.
  Plann. Inference \textbf{106} (2002), 69 -- 76.

\bibitem{GRAPHDESB}
C.~C. Lindner and A.~Tripodi, \emph{The metamorphosis of $\lambda$-fold
  {${K_4-e}$} designs into maximum packings of {${\lambda K_n}$} with
  $4$-cycles}, J. Statist. Plann. Inference \textbf{138} (2008), 3316 -- 3325.

\bibitem{CLIQUEI}
B.~Maheswari and M.~Lavaku, \emph{Enumeration of triangles and hamilton cycles
  in quadratic residue cayley graphs}, Chamchuri J. Math. \textbf{1} (2009), 95
  -- 103.

\bibitem{PEISTWO}
W.~Peisert, \emph{Direct product and uniqueness of automorphism groups of
  graphs}, Discrete Math. \textbf{207} (1999), 189 -- 197.

\bibitem{PEISERT}
\bysame, \emph{All self-complementary symmetric graphs}, J. of Algebra
  \textbf{240} (2011), 209 -- 229.

\bibitem{PALEYPATCHTWO}
H.~Sachs, \emph{{\''U}ber selbstkomplement{\''a}re graphen}, Publicationes
  Mathematicae Debrecen \textbf{9} (1962), 270 -- 288.

\bibitem{CLIQUEH}
C.~Schneider and A.~Silva, \emph{Cliques and colorings of generalized {P}aley
  graphs and an approach to synchronization}, To appear in J. Algebra Appl.

\bibitem{CLIQUEF}
J.~Shearer, \emph{Lower bounds for small diagonal {R}amsey numbers}, Journal of
  Combinatorial Theory, Series A \textbf{42} (1986), 302 -- 304.

\bibitem{INJOURN}
H.~Sun, \emph{On the existence of simple {BIBD}s with number of elements a
  prime power}, Journal of Combinatorial Designs \textbf{21} (2013), no.~2,
  47--59.

\bibitem{THOMSP}
A.~G. Thomason, \emph{Paley graphs and {W}eil's theorem}, Talk at British
  Combinatorial Conference (Unpublished) (1983).

\bibitem{TONCHEV2}
V.~D. Tonchev, \emph{On block designs arising from rank {$3$} graphs}, C. R.
  Acad. Bulgare Sci. \textbf{31} (1978), no.~8, 945--948. \MR{522001
  (80a:05029)}

\bibitem{TONCHEV1}
\bysame, \emph{On block designs arising from rank {$3$} graphs}, J. Statist.
  Plann. Inference \textbf{5} (1981), no.~4, 399--403. \MR{640548 (82m:05021)}

\bibitem{COSPG}
E.~R. van Dam, W.~H. Haemers, and J.~H. Koolen, \emph{Cospectral graphs and the
  generalized adjacency matrix}, Linear Algebra Appl. \textbf{423} (2007),
  no.~1, 33--41. \MR{2312318 (2008b:05097)}

\bibitem{VANWIL}
J.~H. Van~Lint and R.~M. Wilson, \emph{A course in combinatorics}, Cambridge
  University Press; Edition 2, 2001.

\bibitem{WEST}
D.B. West, \emph{Introduction to graph theory}, Prentice Hall; Edition 2, 2001.

\end{thebibliography}
\end{document}